\documentclass[12pt]{amsart}
\usepackage{amscd,amsmath,latexsym}
\usepackage[mathcal]{euscript}
\usepackage[all]{xy}
\CompileMatrices

\newtheorem{thm}{Theorem}[section]
\newtheorem{cor}[thm]{Corollary}
\newtheorem{lem}[thm]{Lemma}
\newtheorem{prop}[thm]{Proposition}
\theoremstyle{definition}
\newtheorem{defin}[thm]{Definition}
\theoremstyle{definition}
\newtheorem{exm}[thm]{Example}
\theoremstyle{remark}
\newtheorem*{rem}{Remark}

\DeclareMathOperator{\colim}{colim}
\DeclareMathOperator{\hocolim}{hocolim}

\newcommand{\Z}{{\mathbb Z}}

\newcommand{\Q}{{\mathbb Q}}
\newcommand{\F}{{\mathbb F}}
\newcommand{\FF}{{\mathcal F}}

\begin{document}

\title[Simplicial spaces of homomorphisms]{Commuting elements, simplicial spaces and filtrations of classifying spaces}

\author[Alejandro Adem]{Alejandro Adem$^{*}$}
\address{Department of Mathematics,
University of British Columbia, Vancouver BC V6T 1Z2, Canada}
\email{adem@math.ubc.ca}
\thanks{$^{~~*}$Partially supported by NSERC}

\author[Frederick R.~Cohen]{Frederick R.~Cohen$^{**}$}
\address{Department of Mathematics,
University of Rochester, Rochester NY 14627, USA}
\email{cohf@math.rochester.edu}
\thanks{$^{~**}$Partially supported by DARPA and the NSF}

\author[Enrique Torres Giese]{Enrique Torres Giese}
\address{Department of Mathematics,
University of Michigan, Ann Arbor MI 48109, USA}
\email{etorresg@umich.edu}

\date{\today}

\begin{abstract}
Let $G$ denote a topological group. In this article the descending
central series of free groups are used to construct simplicial
spaces of homomorphisms with geometric realizations $B(q,G)$
that provide a filtration
of the classifying space $BG$. In particular this setting gives rise to a 
single space constructed out of all the spaces of 
ordered commuting $n$--tuples of elements
in $G$.
Basic properties of these constructions are discussed, including the
homotopy type and cohomology when the group $G$ is either a finite group
or a compact connected Lie group. For a finite group
the construction gives rise to a covering
space with monodromy 
related to a
delicate result in group theory equivalent to the odd-order theorem of
Feit-Thompson. The techniques here also yield a counting formula for the
cardinality of $Hom(\pi, G)$ where $\pi$ is any descending central
series quotient of a finitely generated free group.
Another application is the determination of
the structure of the spaces $B(2,G)$ obtained from commuting $n$-tuples in $G$ 
for finite groups such that the centralizer of every non--central element
is abelian
(known as transitively commutative groups), which played a key role
in work by Suzuki on the structure of finite simple groups.
\end{abstract}

\maketitle
\tableofcontents

\section{Introduction}

Let $G$ denote a topological group. The classifying 
space $BG$ plays a central role in algebraic topology with
important applications to bundle theory and cohomology of groups.
In this paper a filtration of $BG$ is introduced by using the
descending central series of the free groups. A basic feature is
that if $\FF_n$ is the free group on $n$ generators with $\Gamma^q$
the $q$--th stage of its descending central series, then the spaces
of homomorphisms $Hom(\FF_n/\Gamma^q, G)$ can be assembled to form
simplicial spaces with geometric realizations $B(q,G)$ which filter the usual
classifying space $BG$. In other words there are inclusions
$$B(2,G)\subset B(3,G)\subset\dots \subset B(q,G)\subset B(q+1,G)\subset
\dots \subset B(\infty, G)=BG$$
where each term is constructed from
the simplicial spaces associated to terms in the descending central
series of the free group. This naturally gives rise to a functorial
construction on topological groups $G\mapsto B(q,G)$. In fact the
construction provided here affords a principal $G$--bundle
$E(q,G)\to B(q,G)$ which fits into a commutative diagram, namely
there are natural morphisms of principal $G$-bundles
$$\xymatrix{
E(q,G)\ar[r]^{e_q}\ar[d]^p & E(q+1,G)\ar[d]^{p}\ar[r] & EG\ar[d] \\
B(q,G)\ar[r]^{b_q} & B(q+1,G)\ar[r] & BG }$$ and the maps $e_q, b_q$
yield a natural filtration of subspaces for $EG$ and $BG$.

When $G$ is a finite group 
the spaces developed here
give finite, regular covering spaces with structure group $G$.
General properties 
arise from features of the monodromy for these covers; one striking fact
about the structure is that an elementary property of this natural
monodromy is equivalent (see Proposition \ref{prop:on.feit.thompson}) 
to the Feit-Thompson Theorem about
solvability  
for groups of odd order \cite{Feit-Thompson}.
The structure of this monodromy will be
developed more fully elsewhere, but is addressed here in 
the case of 
finite \textsl{transitively commutative} groups, which in the
nonabelian case are characterized by the property
that all
non--central
elements have abelian centralizers. 
These groups were shown to be solvable by Suzuki \cite{Suzuki} when they have
odd order, a seminal result that played a key role in the subsequent proof of the
Feit--Thompson Theorem (see the survey
paper \cite{Solomon},
pages 324--325 for an explanation of this).  

It turns out that the spaces $B(q,G)$ have many interesting features 
which are explored in
this article; the goal is to introduce these objects and describe 
both their similarities and differences when compared to standard
classifying spaces, connecting them whenever possible to basic 
properties of spaces of homomorphisms.
This will involve different aspects of homotopy theory, group cohomology
and group theory. The main focus will be to illustrate these features
in basic examples such as finite groups and connected Lie groups.

The paper is organized as follows.
Section \ref{sec:simplicial.structure} gives proofs of the
simplicial structure of certain spaces of homomorphisms as well as
the construction of $B(q,G)$ as the geometric realization of a
simplicial space. Section \ref{sec:general.properties} gives 
general homological, bundle theoretic and geometric
properties of the spaces $E(q,G)$
and $B(q,G)$.
Section \ref{sec:homotopy} provides a verification that
$B(q,G)$ is a natural colimit which is weakly equivalent
to a more tractable
homotopy colimit. This in turn maps to the classifying space
of the colimit of all
subgroups in $G$ of nilpotence class less than $q$,
with fibre a finite dimensional, possibly
contractible, complex (Theorem \ref{thm:hocolim.to.colim}).
In Section \ref{sec:counting} a natural stable splitting for certain
spaces of homomorphisms is stated, which follows from \cite{ABBCG}
and \cite{AC} under appropriate
conditions.
In the special case of 
a finite group $G$, these stable decompositions count the
cardinality of $Hom(\FF_n/\Gamma^q,G)$ (Corollary 
\ref{cor:cardinality.of.Hom.F.mod.Gamma.q}).
The subject of Section \ref{sec:connected.Lie} is $B(q,G)$
for a connected Lie groups $G$. For $G$ compact this section gives a
computation of the rational cohomology of $B(2,G)$ 
(Theorem \ref{thm:cohomology.Lie}).
It is also shown that the loop space of
$B(q,G)$ admits a natural, non-trivial product decomposition
$G\times\Omega E(q,G)\simeq \Omega B(q,G)$ (Theorem \ref{thm:product.decomposition}).
On the other hand, in case $G$ is finite, the
homological structure of $B(q,G)$ arises as a natural sub-chain
complex of the classical algebraic bar construction. That connection
is developed in Section \ref{sec:cohomology.Feit-Thompson}. In addition, the
first homology group of $B(q,G)$ as well as the monodromy of the
finite cover $E(q,G) \to B(q,G)$ have a close connection to the
Feit-Thompson odd-order theorem, and it is described here (Proposition
\ref{prop:on.feit.thompson}).
This naturally connects to Section \ref{sec:TC.groups} of this paper,
where the structure of $E(2,G)\to B(2,G)$ for transitively
commutative groups is described.
As has been mentioned above, these groups are key examples in finite
group theory. Explicit descriptions are given for these
spaces using an action on a tree (Proposition \ref{prop:TC.tree}). 

The third author would like to thank Fred Cohen for his hospitality
and support given during a stay at the University of Rochester in
Spring 2008.

\section{Simplicial Spaces of Homomorphisms}\label{sec:simplicial.structure}

The goal of this section is to define a family of simplicial spaces
assembled using spaces of homomorphisms. In fact, this family will
yield a filtration of the bar construction of the classifying space
of a group and it will be parametrized by the lower central series
of the free groups.

\begin{defin}
Let $Q$ be a group, define a chain of subgroups $\Gamma^r(Q)$
inductively: $\Gamma^1(Q)=Q$; $\Gamma^{i+1} (Q) = [\Gamma^i(Q),Q]$.
By convention, $\Gamma^\infty(Q) =\{1\}$. The descending central
series of $Q$ is the normal series
$\dots \subset \Gamma^{i+1}(Q)\subset \Gamma^i(Q)\subset  \dots \subset
\Gamma^2(Q)\subset\Gamma^1(Q)=Q$.
\end{defin}

A discrete group $Q$ is said to be \textsl{nilpotent} if there is
some integer $m$ such that $\Gamma^{m+1}(Q)=\{1\}$. The least such
integer $m$ is called the \textsl{nilpotency class} of $Q$. 
For example every
finite $p$--group is nilpotent.

When the context is clear $\Gamma^r$ will be used to denote these
normal subgroups of $Q$. Let $\FF_n = \FF[r_0,r_1,\cdots,r_{n-1}]$
denote the free group on $n$-letters. Fix an integer $q$ and let
$\Gamma^q$ denote the $q$-stage of the descending central series for
$\FF_n$, as above.

Let $G$ denote a topological group; in this paper it will be assumed
throughout that: (i) $G$ is locally compact as well as Hausdorff;
and that (ii) $1 \in G$ is a non-degenerate basepoint. Consider the
set $Hom(\FF_n/\Gamma^q, G)$ of homomorphisms $f\colon \FF_n\to G$ which
descend to homomorphisms $\bar f\colon \FF_n/\Gamma^q \to G$, i.e.
such that $f(\Gamma^q)=1$. An element of $Hom(\FF_n/\Gamma^q, G)$ is
specified by functions from the set $\{r_0,r_1,\cdots,r_{n-1}\}$ to
$G$. Thus an element in $Hom(\FF_n/\Gamma^q, G)$ is identified as an
ordered $n$-tuple $(x_0,x_1,\cdots,x_{n-1})$ with $f(r_i) = x_i$ and
subject to $f(\Gamma^q)=1$.
The space $Hom(\FF_n/\Gamma^q, G)$ is topologized as a subspace of
$G^n$ with the naturally inherited topology for general $G$.

Two simplicial spaces associated to these homomorphisms will now be
introduced.

\begin{defin}\label{defin:E.q.G}
Let
$$E_n(q,G)=G\times Hom(\FF_n/\Gamma^q,G)\subset G^{n+1},$$
and define $d_i:E_n(q,G)\to E_{n-1}(q,G)$ for $0\le i\le n$ and
$s_j:E_n(q,G)\to E_{n+1}(q,G)$, for $0\leq j\leq n$, given by

$$d_i(g_0,\ldots,g_n)=\left\lbrace\begin{array}{lr}
(g_0,\ldots,g_i\cdot g_{i+1},\ldots,g_n) & 0\leq i<n \\
(g_0,\ldots,g_{n-1}) & i=n\\
 \end{array}\right. $$
and
$s_j(g_0,\ldots,g_n)=(g_0,\ldots,g_i,1,g_{i+1},\ldots,g_n)$
for $0\le j\le n$.
\end{defin}

\begin{defin}\label{defin:B.q.G}
Similarly, let
$B_n(q,G)=Hom(\FF_n/\Gamma^q,G)$
with maps $d_i$ and $s_j$ defined in the same way, except that the
first coordinate $g_0$ is omitted and the map $d_0$ takes the form
$d_0(g_1,\ldots,g_n)=(g_2,\ldots,g_n)$.
\end{defin}

\begin{lem}\label{lem:simplicial.space} The maps $d_i$, $s_j$ defined on the spaces $E_n(q,G)$ and
$B_n(q,G)$ are well-defined and equip them with the structure of
simplicial spaces.
\end{lem}
\begin{proof} The maps $d_i$ and $s_i$ are well-defined since they are
induced by group homomorphisms between free groups and from the fact
that if $f:A\to B$ is a group homomorphism then
$f(\Gamma^qA)\subseteq \Gamma^qB$. The simplicial identities follow
(as in the classical case for the bar construction) from the
definition of the homomorphisms inducing the maps $d_i$ and $s_j$.
\end{proof}

Note that the map $G^{n+1}\to G^n$ that projects the last $n$
coordinates onto $G^n$ defines a simplicial map
$p_*:E_*(q,G)\to B_*(q,G)$.
Moreover, $G$ acts from the left on $E_n(q,G)$ by multiplication on
the first coordinate
$g(g_0,g_1\ldots,g_n)=(gg_0,g_1,\ldots,g_n)$
and this action makes $E_*(q,G)$ into a $G$-simplicial space; that
is, the action of $G$ commutes with the face and degeneracy maps.
This action is free and its degree--wise orbit space is homeomorphic
to $B_*(q,G)$ \cite{milgram}.
Next recall the notion of geometric realization for a simplicial space
\cite{milnor}.

\begin{defin}
The geometric realization of a simplicial space $Z_*$ is the
following topological space
$|Z_*| := \coprod_{n\ge 0} Z_n \times \Delta^n / \sim $,
where $\Delta^n$ denotes the $n$--simplex and the equivalence
relation $\sim$ is defined as follows. Identify $(x,\delta_it)\in
X_n\times\Delta^n$ with $(d_ix,t)\in X_{n-1}\times \Delta^{n-1}$ for
any $x\in X_n$, $t\in \Delta^{n-1}$ and $(x,\sigma_jt) \in
X_n\times\Delta^n$ with $(s_jx, t) \in X_{n+1}\times \Delta^{n+1}$
for any $x\in X_{n-1}$ and $t\in \Delta^{n+1}$. The topology on
$|Z_*|$ is the quotient topology.
\end{defin}

\begin{defin}
The realization $|Z_*|$ of a simplicial space has a canonical
filtration defined as follows: $F_r|Z_*|\subset |Z_*|$ is the image
of $\coprod_{0\le i\le r} Z_i\times\Delta^i$ in the quotient space.
The associated graded space is defined as $E^0_j (|Z_*|) =
F_j|Z_*|/F_{j-1}|Z_*|$.
\end{defin}

\begin{exm}\label{exm:EG.to.BG}
In the case of Definitions \ref{defin:E.q.G}, and \ref{defin:B.q.G}
with $q = \infty$, $|E_*(\infty,G)|=EG$, $|B_*(\infty,G)|=BG$ and
$p$ is the standard map $EG\to BG$. In other words, this yields
Milgram's model for the universal principal $G$--bundle $EG\to BG$
(\cite{MacLane}, \cite{Steenrod}, and \cite{AM}, page 49). Notice
that the spaces $E_*(q,G)$ and $B_*(q,G)$ are simplicial subspaces
of the bar construction of $G$.

\end{exm}

\begin{defin} The geometric realizations of $E_*(q,G)$ and $B_*(q,G)$
are denoted by $E(q,G)$ and $B(q,G)$ respectively.
\end{defin}

Note that the map $p:E(q,G)\to B(q,G)$ induced by $p_*$ is a principal
$G$--bundle as it can be thought of as the natural pullback of the
principal bundle $EG\to BG$ using the inclusion $B(q,G)\to BG$
(see \cite{Steenrod}, page 364, Theorem 8.3).
The following definition will be used later and is useful for the
analysis of the local structure of the map $p$ on the level of
geometric realizations.

\begin{defin}\label{defin:filtrations} Let $S_n(j,q,G)$ be the subspace of 
$Hom(\FF_n/\Gamma^q,G)$ obtained by taking
$n$-tuples with at least $j$ coordinates equal to $1_G$. 
\end{defin}

\begin{prop} Let $G$ be a topological group.
\begin{itemize}
\item[1)] The natural surjection $\FF_n/\Gamma^{q+1}\to \FF_n/\Gamma^q$
induces a map of simplicial spaces compatible with the simplicial
maps $p_n$; that is, there is a commutative diagram
$$\xymatrix{
E_n(q,G)\ar[r]^{e_q}\ar[d]^{p_n} & E_n(q+1,G)\ar[d]^{p_n} \\
B_n(q,G)\ar[r]^{b_q} & B_n(q+1,G) }.$$ 
\item[2)] There are natural morphisms of principal $G$-bundles
$$\xymatrix{
E(q,G)\ar[r]^{e_q}\ar[d]^p & E(q+1,G)\ar[d]^{p}\ar[r] & EG\ar[d] \\
B(q,G)\ar[r]^{b_q} & B(q+1,G)\ar[r] & BG }$$
\item[3)] The maps $e_q, b_q$ yield a natural filtration of subspaces for
$EG$ and $BG$.
\item[4)] If $G$ is a nilpotent group of class $c$,
then $E(c+1,G)=EG$, $B(c+1,G)=BG$ and the filtration is
finite. If $G$ is finitely generated, then the filtration is finite
if and only if $G$ is nilpotent.
\end{itemize}
\end{prop}
\begin{proof}
Part (1) and (2) are immediate since the maps $e_q, b_q$ are induced
by group homomorphisms, whereas part (3) follows from (2). If $G$ is
nilpotent of class $c$, then $\Gamma^{c+1}(G)=\{ 1\}$ and thus
$Hom(\FF_n/\Gamma^{c+1},G)=G^n$ for all $n$. Hence $E(c+1,G)=EG$ and
$B(c+1,G)=BG$. If $B(q,G)=BG$ for some $q$, then for all $m$
$$F_mB(q,G)-F_{m-1}B(q,G)=F_mB(\infty,G)-F_{m-1}B(\infty,G).$$
The relations imposed by the face and degeneracy maps show that
$$F_mB(q,G)-F_{m-1}B(q,G)=
(Hom(\FF_m/\Gamma_q,G)-S_m(1,q,G))\times(\Delta_m-\partial\Delta_m)$$
so the formula above implies that
$$Hom(\FF_m/\Gamma_q,G)-S_m(1,q,G)=Hom(\FF_m,G)-S_m(1,\infty,G).$$

Suppose that $G$ is generated by $m$ elements and that $\phi:\FF_m\to
G$ is a homomorphism onto the generators of $G$. 
Note that $\phi\in Hom(\FF_m,G)-S_m(1,\infty,G)$. Therefore
$\Gamma^q(G)=\phi(\Gamma^q)=\{ 1\}$ and $G$ is nilpotent.
\end{proof}

\section{Further Properties of 
$Hom(\FF_{n}/\Gamma^{r}, G)$, $E(q,G)$ and $B(q,G)$}\label{sec:general.properties}

The functor from topological groups to topological spaces, given by
$G\mapsto B(q,G)$
has a number of interesting features, some of them analogous to the
classifying space functor but there are some important differences.

\begin{prop}\label{prop:general.properties} The functor $B(q,G)$
satisfies the following properties:

\begin{enumerate}

\item If $q \geq 2$ and $H$ is a topological group of
nilpotency class less than $q$, then $B(q,H) = BH$, the usual
classifying space.

\item If $G$ is a finite group then there exists
an $N$ that depends on $G$ such that $B(q,G)=B(N,G)$ for all $q\ge
N$.

\item If $q \geq 2$ and $\iota\colon H \to G$ is a homomorphism
where $H$ is a group of nilpotency class less than $q$, then there
is a commutative diagram
\[
\begin{CD}
 BH @>{1}>> BH      \\
 @VV{}V    @VV{B(\iota)}V       \\
B(q,G) @>{}>> BG.\\
\end{CD}
\] Thus if $BG$ is a stable retract of $BH$, then
$BG$ is a stable retract of $B(q,G)$.

\item If $q \geq 2$ and $Z$ is a central subgroup of $G$, then 
the multiplication
map $\mu: Z \times G \to G$ induces an action
$\mu:
BZ \times B(q,G)\to  B(q,G)$ together with a commutative diagram
\[
\begin{CD}
BZ \times BZ @>{}>> BZ     \\
 @VV{}V    @VV{I}V       \\
BZ \times B(q,G) @>{\mu}>> B(q,G).\\
\end{CD}
\]

\item The
functor $B(q, G)$ commutes with direct limits in $G$, i.e.
there are homeomorphisms
$$B(q,\varinjlim_{\alpha} G_{\alpha})
\cong \varinjlim_{\alpha} B(q, G_{\alpha}).$$

\end{enumerate}

\end{prop}

\begin{proof} These properties are easy to verify, in particular the stable
value $B(N,G)$ for the spaces $B(q,G)$ can be explained as follows.
Let $N$ denote the minimal integer such that the class of nilpotence
for any nilpotent subgroup of $G$ is strictly less than $N$. It
suffices to show that
$Hom(\FF_n/\Gamma^N, G)=Hom(\FF_n/\Gamma^{N+t},G)$
for any $t\ge 1$. Given a homomorphism $\phi : \FF_n\to G$, where
$\FF_n$ is a free group of rank $n$, let $H=\rm{Im}\phi$. Then $\phi$
will descend to a map $\FF_n/\Gamma^{N+t}\to G$ if and only if the
nilpotence class of $H$ is less than $N+t$. However, given that $H$
is nilpotent, its nilpotence class is in fact less than $N$; hence
$\phi$ represents a unique element in $Hom(\FF_n/\Gamma^N,G)$ and the proof
is complete. Note that if $G$ is nilpotent then the stable value 
of $B(q,G)$ is
$BG$.
\end{proof}

Recall that the map $c_g:BG\to BG$ induced by conjugation by $g\in
G$ is homotopic to the identity. It turns out that for $B(q,G)$ this
need not be the case. The reason is that $B_*(q,G)$ splits into a
lattice of nerves of categories so that conjugation may permute this
lattice. On the other hand, this difference indicates that when
comparing the cohomology of $BG$ with that of $B(q,G)$ it is only
necessary to look at the invariant part of the cohomology of
$B(q,G)$ as the following diagram
\[ \xymatrix{
B(q,G)\ar[r]\ar[d]^{c_g} & BG\ar[d]^{c_g} \\
B(q,G)\ar[r] & BG } \] commutes. Hence the map $H^*(BG;\Z)\to
H^*(B(q,G);\Z)$ factors through $H^*(B(q,G);\Z)^G$.

The next proposition deals with basic facts about the cohomology of
the spaces $B(q,G)$. These have close connections to properties of
ordinary group cohomology.

\begin{prop}\label{prop:general.cohomological.properties} The functors
$B(q,G)$ satisfy the following properties:

\begin{enumerate}

\item If $G$ is a finite group with mod $p$ cohomology detected
by subgroups of nilpotence class less than $q$, then the map
induced on mod $p$ cohomology $$H^*(BG;\mathbb
F_p)\to H^*(B(q,G);\mathbb F_p)$$ 
is a monomorphism.

\item If $G$ is a finite group, then for any
$q\ge 2$, the map induced on mod-$p$ cohomology
$$H^*(BG;\mathbb F_p) \to H^*(B(q,G);\mathbb F_p)$$ 
is a
monomorphism modulo nilpotent elements.

\item The spectral sequence obtained from the skeletal filtration
satisfies $$E^1_{s,t} =  \oplus H_{s+t}(F_sB(q,G),F_{s-1}B(q,G);
\mathbb Z)$$ and abuts to $H_{s+t}(B(q,G);\mathbb Z).$ In case $q
=\infty$, and coefficients are in a field $\mathbb F$, this spectral
sequence is the bar spectral sequence with
$E_{s,t}^2 =  {Tor_{s,t}}^{H_*(G)}(\mathbb F,\mathbb
F)$ which abuts to $H_{s+t}(BG;\mathbb F)$.
\end{enumerate}

\end{prop}

\begin{proof}
The first statement follows immediately from parts (1) and (2) of
the previous proposition. For the second statement it suffices to
prove it for $q=2$. Recall the result due to Quillen and Venkov,
namely that if $A_p(G)$ denotes the poset of non--trivial elementary
abelian $p$--subgroups of $G$ with morphisms induced by inclusion
and conjugation, then the map induced by the restrictions
$Q: H^*(BG;\mathbb F_p) \to \lim_{E\in A_p(G)} H^*(E;\mathbb F_p)$
has nilpotent kernel (see \cite{QP} and \cite{AM}, page 146). Note,
however, that this factors through the invariants of the mod $p$
cohomology of $B(q,G)$ for any $q\ge 2$; hence any element in the
kernel of $\phi^*:H^*(BG;\mathbb F_p)\to H^*(B(2,G);\mathbb F_p)$
must be nilpotent, as it is necessarily in the kernel of the map
$Q$.

The statement about the spectral sequence is routine and the
identification with the bar spectral sequence in the classical case
is well--known (see \cite{mccleary}).
\end{proof}

Instead of considering the descending central series of a group, it
is also possible to consider its \textsl{$p$--descending central
series}, defined as follows:

\begin{defin}
Let $p$ denote a prime number and $Q$ a group; define a chain of
subgroups $\Gamma^r_p(Q)$ inductively: $\Gamma^1_p(Q)=Q$;
$\Gamma^{i+1}_p(Q)=[\Gamma^i_p(Q), Q](\Gamma^i_p(Q))^p$. By
convention, $\Gamma^{\infty}_p(Q)=\{1\}$. The $p$--descending
central series of $Q$ is the normal series
$$\dots \subset \Gamma^{i+1}_p(Q)\subset
\Gamma^i_p(Q)\subset\dots\subset \Gamma^2_p(Q)\subset
\Gamma^1_p(Q)=Q.$$
\end{defin}

A discrete group $Q$ is said to be $p$--nilpotent of there is some
integer $m$ such that $\Gamma^{m+1}_p(Q)=\{1\}$. The least such
integer is called the $p$--class of $Q$. Given the free group $\FF_n$
on $n$ letters, denote by $\Gamma^q_p$ the $q$--th stage of the
descending central series for $\FF_n$.
There are some important differences here, in particular
$\Gamma^i_p/\Gamma^{i+1}_p$ is an elementary abelian $p$--subgroup
and the quotient groups $\FF_n/\Gamma^i_p$ are all finite $p$--groups.
Just as before, they can be assembled to yield simplicial spaces
$E_*(q,G,p)$ and $B_*(q,G,p)$ where $G$ is a topological group. The
basic fact used is that if $f:A\to B$ is a group homomorphism then
$f(\Gamma^i_p(A))\subset \Gamma^i_p(B)$. We can take their geometric
realization and thus obtain spaces $E(q,G,p)$ and $B(q,G,p)$; they
will now exhibit properties specific to the prime $p$. For example
$B(2,G,p)$ is a space assembled from the $p$--elementary abelian
subgroups in $G$. Here is an example of how these spaces capture
$p$--local information.

\begin{prop}
Let $G$ denote a finite group such that its mod $p$ cohomology is
detected on $p$--elementary abelian subgroups; then the map
$B(2,G,p)\to BG$ induces an injective map in mod $p$ cohomology.
\end{prop}
The proof follows from restricting to elementary abelian subgroups.
From the inclusions
$$B(2,G,p)\subset B(2,G)\subset BG$$
it follows that detecting on $B(2,G,p)$ is sharper than doing it with
$B(2,G)$.

\begin{exm}
Let $G=\Sigma_n$, the symmetric group on $n$ letters. Then it is
well known (see \cite{AM}, Ch. VI) that its mod $2$ cohomology is
detected on elementary abelian $2$--groups. Thus it follows that for
all $n\ge 1$
$H^*(\Sigma_n; \mathbb F_2)
\hookrightarrow H^*(B(2,\Sigma_n,2);\mathbb F_2)$ is a
monomorphism.
\end{exm}

\begin{rem}
Other variations of these constructions can be obtained by using
the free pro--$p$--groups or the profinite completions of the
free groups. Likewise taking quotients under the natural
conjugation action produces a simplicial space with terms
of the form $Rep(\mathbb Z^n, G)$. These constructions 
have interesting properties but will be considered
elsewhere.
\end{rem}

\section{Homotopy Properties of $B(q,G)$ for $G$ a Finite Group}\label{sec:homotopy}

This section primarily addresses the case when $G$ is a finite
group.

\begin{defin}
For a finite group $G$, let
$\mathcal{N}_q(G) = \{ A\subset G ~\rm{subgroup}
~|~ \Gamma^{q}(A)=\{1\}\}$
\end{defin}

The elements of $\mathcal{N}_q(G)$ are precisely all the subgroups
of nilpotence class less than $q$. Note that this is a partially
ordered set under inclusion, furthermore it is closed under
conjugation by elements in $G$ and under the process of taking
subgroups (indeed if $H\subset K$ is a subgroup, then
$\Gamma^i(H)\subset \Gamma^i(K)$ for all $i\ge 1$). Such a
collection is often referred to as a \textsl{family} of subgroups in
$G$. A group can be defined associated to these subgroups and the
inclusions between them.

\begin{defin}\label{colimit}
If $G$ is a finite group, let
$G(q)=\underset{A\in\mathcal{N}_q(G)}\colim A$.
\end{defin}

Note that
$\mathcal{N}_2(G)$ is the family of abelian subgroups in $G$ and
$G(2)$ is simply the colimit of all abelian subgroups in $G$. 
The classifying spaces of the groups $G(q)$ will play a role
in the analysis of the spaces $B(q,G)$.

\begin{thm}\label{B(q,G)colimit}
Let $G$ be a finite group, then for any $q\ge 2$,
$$B(q,G) = \underset{A\in \mathcal{N}_q(G)}\colim BA.$$
\end{thm}
\begin{proof}
For any subgroup $H\subset G$, there are natural inclusions
$$Hom(\FF_n/\Gamma^q, H)\subset Hom(\FF_n/\Gamma^q, G).$$ 
Given any
homomorphism $\psi: \FF_n \to G$ with $\psi (\Gamma^q)=1$, it follows
that $K=\rm{Im}~\psi$ belongs to $\mathcal{N}_q(G)$, and that
$\psi\in Hom(\FF_n/\Gamma^q, \rm{Im}~\psi)$. Now if $A\in
\mathcal{N}_q(G)$, then $Hom(\FF_n/\Gamma^q, A) = Hom (\FF_n, A)$ for
all $n\ge 0$; combining these facts yields $B_*(q,G)
=\bigcup_{A\in\mathcal{N}_q(G)} B_*(\infty, A)$. This simplicial
space is contained in $B_*(\infty, G)$, and the realizations $BA$
are natural subspaces of $BG$. If $A\subset A'$ then $BA\subset
BA'$; moreover $BA\cap BA' = B(A\cap A')$. Therefore the realization
$|\bigcup_{A\in\mathcal{N}_q(G)} B_*(\infty, A)|$ of the union is
precisely the space that results from identifying the realizations
$BA$ along their intersections i.e. the colimit, whence the result
follows.
\end{proof}

Note that the classifying
space functor does not in general commute with colimits.
In any case, since colimits are sometimes delicate to handle, 
it is preferable to
work with a \textsl{homotopy colimit} whenever possible (see
\cite{WZZ}). Under favorable circumstances they are weakly homotopy
equivalent. This will occur if the diagram of spaces is a
\textsl{free diagram} (see \cite{DF}); for a union this will hold if
all the intersections of spaces in the diagram are used in the
construction. This is automatically verified for the diagrams under
consideration here, because of the nice properties of the functor
$A\mapsto BA$ which have been outlined above.

Thus there is a weak equivalence $B(q,G) \simeq
\underset{A\in\mathcal{N}_q(G)}\hocolim~BA$.  
Now the identity element in $G$ makes all of the spaces in these diagrams
\textsl{pointed spaces} (i.e. the classifying space $B\{1\}$ is the
natural basepoint $*$). From this it follows that the fundamental
group of $B(q,G)$ can be computed as a colimit
$\pi_1 (B(q,G),*)\cong \underset{A\in \mathcal{N}_q(G)}\colim A\cong G(q)$.
This can be expressed more succinctly as follows. Let $\{M_i ~|~
i\in I\}$ denote a collection of maximal subgroups in
$\mathcal{N}_q(G)$. Then
$\pi_1(B(q,G),*)\cong *_{i\in I} M_i /(r_{ij})$
where the $r_{ij}$ are the relations coming from the inclusions
$M_i\cap M_j\subset M_i$.

Having established that $B(q,G)$ is a homotopy colimit, there are
some well-known methods from homotopy theory that can be used to
study its homotopy type. 

\begin{thm}\label{thm:hocolim.to.colim}
There is a natural fibration $B(q,G)\to BG(q)$ with fiber a
simply--connected finite dimensional complex $K_q$.
\end{thm}
\begin{proof}
If $A\in \mathcal{N}_q(G)$, then there is a natural inclusion
$A\subset G(q)$. This gives rise to a fibration $G(q)/A \to BA\to
BG(q)$, where the fiber is a discrete coset space. Taking homotopy
colimits and applying the results in \cite{P} (see \cite{DF}, page
180) yields a fibration

$$\underset{A\in \mathcal{N}_q(G)}\hocolim~ G(q)/A
\to \underset{A\in \mathcal{N}_q(G)}\hocolim~ BA \to BG(q)$$ and as
the fiber $K_q$ is a homotopy colimit of a finite collection of
discrete spaces it is necessarily finite dimensional. On the other
hand $BG(q)$ is a $K(\pi, 1)$ and the map $B(q,G)\to BG(q)$ induces
an isomorphism on fundamental groups, whence the result follows.
\end{proof}

A basic question is that of determining under what conditions the
fibre $K_q$ is contractible.

\medskip

\noindent\textbf{Question}: \emph{If $G$ is a finite group, are the
spaces $B(q,G)$ Eilenberg--Mac Lane spaces of type $K(G(q), 1)$?}

For any $q\ge 2$, the group $G(q)$ admits a natural surjection onto
$G$, as every element of $G$ is contained in an abelian subgroup,
and these groups all belong to $\mathcal{N}_q(G)$ for $q\ge 2$. i.e.
there is a group extension
$1\to T(q) \to G(q) \to G\to 1$.
The kernel is torsion--free by construction, as every finite group
in $G(q)$ necessarily embeds in $G$. Thus there are no exotic finite
groups appearing in the colimit. For all $A\in \mathcal{N}_q(G)$
the group $T(q)$ acts freely on
the coset space $G(q)/A$, with quotient precisely $G/A$. Applying
homotopy colimits, the fibrations $G(q)/A \to G/A \to BT(q)$ give
rise to a fibration $K_q\to
\underset{A\in\mathcal{N}_q(G)}\hocolim~ G/A\to BT(q)$. On the
other hand, as in the proof of Theorem \ref{thm:hocolim.to.colim}, 
the fibrations $G/A\to BA\to BG$ give rise to 
fibrations $\underset{A\in\mathcal{N}_q(G)}\hocolim~ G/A\to
B(q,G)\to BG$. Comparing this to the natural fibration $E(q,G)\to
B(q,G)\to BG$ it can be seen that there is an equivalence
$E(q,G)\simeq \underset{A\in\mathcal{N}_q(G)}\hocolim~ G/A$
and that $\pi_1(E(q,G))\cong T(q)$, which is a torsion--free group.
Note that the \textsl{hocolim} term is a finite complex. These
spaces can be assembled into a commutative diagram of fibrations:

\[
\xymatrix{
 K_q \ar@{=}[r] \ar[d] & K_q \ar[d]  \\
 E(q,G) \ar[r] \ar[d] & B(q,G) \ar[r] \ar[d] & BG
\ar@{=}[d]      \\
 BT(q) \ar[r] & BG(q)\ar[r]
 & BG  \\
}
\]
By the definition of $K_q$ in Theorem \ref{thm:hocolim.to.colim},
$K_q$ is contractible if and only if $B(q,G)\simeq BG(q)$; this
property is equivalent to $E(q,G)\simeq BT(q)$.

The complexes $K_q$ are defined as homotopy colimits
of functors with values in $G(q)$--spaces of the form
$G(q)/A$ where $A$ is a finite group, and so admit a 
natural $G(q)$ action with finite stabilizers
which restricts to a free action of the subgroup $T(q)$. The orbit
space $K_q/T(q)$ is a finite complex with a natural $G$--action, and
it is homotopy equivalent to $E(q,G)$. Now taking the quotient by
$G(q)$ yields the equivalence
$K_q/G(q)\simeq \underset{A\in\mathcal{N}_q(G)}\hocolim~ \{*\}
\simeq |\mathcal{N}_q(G)|\simeq~\{*\}$ due to the fact that the
nerve of the category is contractible, as it has a minimal element
(the trivial subgroup $\{1\}$). On the other hand, there is an
identification
$B(q,G)\simeq EG(q)\times_{G(q)}K_q$ which arises from the
equivalence
$$\underset{A\in\mathcal{N}_q(G)}\hocolim~EG(q)\times_{G(q)}G(q)/A
\simeq EG(q)\times_{G(q)}\underset{A\in\mathcal{N}_q(G)}\hocolim~G(q)/A.$$ 
Note that the projection map $B(q,G)\to |\mathcal{N}_q(G)|$ is a
rational equivalence (as the isotropy is finite) and hence it
follows that $H^i(B(q,G);\mathbb Q)\cong 0$ for $i>0$ (this
also follows from Theorem \ref{B(q,G)colimit}).

\begin{defin}
Let $G$ denote a finite group and let $\mathcal{P}_q(G)$ be the
category with objects the set
$\{ M_\alpha, M_\alpha\cap M_\beta ~|~ M_\alpha, M_\beta\in
\mathcal{M}_q(G)\}$ where $\mathcal{M}_q(G)$ denotes the set of
maximal subgroups in $G$ of nilpotence class $<q$, and the morphisms are the
set of inclusions of the form $M_\alpha\cap M_\beta \to M_\alpha. $
\end{defin}

This category is $1$-dimensional, as there are no compositions. This
can be identified with a \textsl{graph of groups}, that is, an
oriented graph with a group at each vertex and a homomorphism
between the two vertices of each edge according to orientation. Note
that in this situation the graph is always connected.

\begin{thm}\label{thm:aspherical B(q,G)}
Let $q\ge 2$ and $G$ a finite group such that $\mathcal{P}_q(G)$ is
a tree. Then the space $B(q,G)$ is aspherical and there is a natural
homotopy equivalence
$B(q,G)\simeq BG(q)$.
\end{thm}
\begin{proof}
Using the graph of groups $\mathcal{P}_q(G)$ a space $\mathbb
B\mathcal{P_q(G)}$ can be constructed by inserting a copy of the
classifying space of each group at each vertex of the graph and by
filling in a mapping cylinder for each map induced on classifying
spaces on each edge (in this case arising from subgroup inclusions).
The two ends of the mapping cylinder are identified with the
classifying spaces on the respective vertices. This is a special
case of a general construction for graphs of groups; in fact by
\cite{hatcher}, Theorem 1B.11 it follows that $\mathbb
B\mathcal{P}_q(G)$ is aspherical (the relevant hypothesis is the
injectivity of all the homomorphisms on edges).

By construction the maps between classifying spaces have been
replaced by cofibrations. Given that $\mathcal{P}_q(G)$ is a tree,
the space $\mathbb B\mathcal{P}_q(G)$ is homotopy equivalent to
their colimit (by collapsing the mapping cylinders) and this is
precisely $B(q,G)$.
\end{proof}

It is interesting to note that in the situation above, the colimit
group $G(q)$ can be readily understood using the theory of trees as
in \cite{Serre}; the special condition can be described as those
which determine a \emph{tree of groups}, and $G(q)$ is the
corresponding inductive limit associated to the tree.

\begin{exm}
For the symmetric group $\Sigma_4$ 
the diagram of abelian subgroups
is quite intricate although it can be shown that in fact
$B(2,\Sigma_4)$ is aspherical. On the other hand the maximal
subgroups of class less than $3$ in $\Sigma_4$ are isomorphic to
either $D_8$ or $\mathbb Z/3$, and every pair of distinct copies of
$D_8$ intersect along the subgroup
$$K= \{1, (12)(34), (13)(24), (23)(14)\}\cong
\mathbb Z/2\times\mathbb Z/2.$$ Furthermore,
$\mathcal{N}_q(\Sigma_4)=\mathcal{N}_3(\Sigma_4)$ for all $q\ge 3$,
hence it follows from the theorem above that
$$B(q,\Sigma_4)\simeq \bigvee^4 B\mathbb Z/3\vee
B(*_K^3 D_8)\simeq B\underset{A\in\mathcal{N}_q(\Sigma_4)}\colim~A$$
\end{exm}

\begin{rem}
The spaces $B(q,G,p)$ can also be analyzed using the methods
discussed in this section. In particular it is not hard to show that
$B(q,G,p) =
\underset{K\in\mathcal{S}_p(G)}\colim~ B(q,K,p)$ where
$\mathcal{S}_p(G)$ is the poset of $p$--subgroups in $G$. This in
turn can be further decomposed using the the classifying spaces of
$p$--subgroups which are of $p$--nilpotence class less than $q$.
\end{rem}

\section{Stable Splittings and a Counting Formula for Homomorphisms}
\label{sec:counting}

A stable splitting was given for the space of commuting $n$--tuples
in \cite{AC}. Analogous stable splittings arise for
$Hom(\FF_n/\Gamma^q,G)$ from the fat wedge filtration of the product
$G^n$ where the base-point of $G$ is $1_G$.

\begin{defin}\label{defin:simplicial.spaces}
As before, let $S_{n}(j,q,G)$ denote the subspace consisting of
elements in  
$Hom(\FF_n/\Gamma^q,G)$ with at least
$j$ coordinates equal to $1_G$.
Let $S_n(q,G)$ denote $S_n(1,q,G)$.
A Lie  group $G$ is said have \textsl{cofibrantly filtered elements}
if the natural inclusions $I_j: S_{n}(j,q,G)\to S_{n}(j-1,q,G)$ are
cofibrations for all $n$, $q$ and $j$ for which both spaces are
non--empty.
\end{defin}

\noindent It would seem that many Lie groups $G$ should have
cofibrantly filtered elements; it is plausible to conjecture that
this holds if $G$ is a closed subgroup of $GL(n,\mathbb C)$. Note
that in \cite{AC} the weaker condition of having cofibrantly
commuting elements (i.e. the special case $q=2$) was indeed verified
for these groups. The following result describes the stable
structure of the spaces of homomorphisms $Hom(\FF_n/\Gamma^q,G)$ in
terms of more recognizable pieces.

\begin{prop} \label{prop:stable.decompositions.for.general.G}
If $G$ has cofibrantly filtered elements, then there are homotopy
equivalences
$$\Sigma  Hom(\FF_n/\Gamma^q,G)\to
\bigvee_{1 \leq k \leq n}\Sigma  \bigvee^{\binom n k}
Hom(\FF_k/\Gamma^q,G)/ S_k(q,G)$$
and the natural filtration
quotients 
$$E^0_k(B(q,G))= F_kB(q,G)/F_{k-1}B(q,G)$$ 
of the geometric
realization $B(q,G)$ are stably homotopy equivalent to the summands
$$Hom(\FF_k/\Gamma^q,G)/S_k(q,G).$$
\end{prop}

This result is a special case of very general
splitting for simplicial spaces $X_*$ which are \textsl{proper} and
\textsl{simplicially NDR}. Details of this appear in \cite{ABBCG};
the main work is to verify these conditions for the simplicial
spaces $B_*(q,G)$.

In the special case when $G$ is a finite
group, the spaces
of homomorphisms are finite sets, and the cofibrantly filtered
condition is easily verified. The decomposition above can
be interpreted as a numerical formula. The following definition will
be helpful for keeping track of the numbers which will appear in 
counting the cardinality of these spaces of homomorphisms.

\begin{defin}\label{defin:sizes.of.hom.F.mod.gamma.q}
Let $G$ be a finite group; the integer $\lambda_n(q,G)$ is
defined as the cardinality of $Hom(\FF_n/\Gamma^q,G)$, and the integer
$\mu_k(q,G)$ is defined as the rank of $H_0(E^0_k(B(q,G);\mathbb
Z)$.
\end{defin}
\noindent An immediate consequence of Theorem
\ref{prop:stable.decompositions.for.general.G} is the formula
\begin{cor} \label{cor:cardinality.of.Hom.F.mod.Gamma.q}
If $G$ is a finite group, then $$\lambda_n(q,G) = 1 + \sum_{1 \leq k
\leq n} {\binom n k} \mu_k(q,G).$$
\end{cor}

The special case $q=2$ is especially interesting as the formula
above provides information on the cardinality of the set of
commuting elements in a finite group.

\begin{exm}
If $A$ denotes a finite abelian group, then this formula
reduces to
$$|A|^n = 1 + \sum_{1 \leq k
\leq n} {\binom n k} (|A|-1)^k.$$
\end{exm}

\begin{exm}\label{exm:TC.trivial.center.abelian.p.sylow.rank.of.commuting.n.tuples}
Consider the case when $G$ is a finite transitively commutative
group (see Section \ref{sec:TC.groups} for their properties) 
with trivial center. 
The space $B_*(2,G)$ is the
one--point union of the simplicial spaces $B_*C_G(a_i)$, $1\le i\le
N$, where each $C_G(a_i)$ is a maximal abelian subgroup and so in
this case
$\mu_k(2,G)=\sum_{1\le i\le N}(|C_G(a_i)|-1)^k$
and
$$\lambda_n(2,G)= 1 + \sum_{1 \leq k
\leq n} {\binom n k}\sum_{1\le i\le N}(|C_G(a_i)|-1)^k.$$
This applies to $G=A_5$, the alternating group; there are three
isomorphism classes of centralizers: $\mathbb Z/2\times\mathbb Z/2$
(five copies), $\mathbb Z/3$ (ten copies) and $\mathbb Z/5$ (six
copies). This yields
$$\lambda_n(2,A_5) =1 + \sum_{1 \leq k
\leq n} {\binom n k} [5\cdot 3^k+10\cdot 2^k + 6\cdot 4^k]$$
\end{exm}

\section{The Spaces $B(q,G)$ for Connected Lie Groups}\label{sec:connected.Lie}

In this section basic properties of the spaces $B(q,G)$ will be
analyzed for connected Lie groups, including a
calculation of the rational cohomology of $B(2,G)$ when $G$ is
compact.

For a compact, connected Lie group $G$, consider the map
$$\phi_n:G/T\times T^n\to Hom(\Z^n,G)$$
given by
$$(gT,t_1,\ldots,t_n)\mapsto(gt_1 g^{-1},\ldots,gt_n g^{-1}).$$
As described in \S 2, an element in
$Hom(\Z^n,G)$ is represented as an ordered $n$--tuple of
commuting elements in $G$.
Notice that the Weyl group $W(G)=N(T)/T$ acts freely on $G/T\times
T^n$ by
$$(gT,t_1,\ldots,t_n)\cdot w=(gwT,w^{-1}t_1w,\ldots,w^{-1}t_nw)$$ 
and
that the map $\phi_n$ is invariant under this action. This yields
the following commutative diagram
$$\xymatrix{
G/T\times T^n\ar[r]^{\phi_n}\ar[d] & Hom(\Z^n,G) \\
(G/T\times T^n)/W(G)\ar[ru]^{\overline{\phi}_n} & }$$ Note that the
vertical map is equivalent to the map $(G\times T^n)/T\to (G\times
T^n)/N(T)$. Assume that the space $Hom (\Z^n, G)$ is
path--connected, then as was shown in \cite{baird}, $\phi_n$ is
surjective and the fibers of the induced map
$\overline{\phi_n}: (G\times T^n)/N(T)\to Hom (\Z^n,G)$
are rationally acyclic. Hence applying the Vietoris-Begle Theorem it
follows that $\overline{\phi_n}$ induces a rational homology
equivalence and thus the map $\phi_n$ induces an isomorphism
$H^*(Hom(\Z^n,G);\Q)\to H^*(G/T\times
T^n;\Q)^{W(G)}$. This information can be assembled to yield the
following theorem.

\begin{thm}\label{thm:cohomology.Lie} 
Let $G$ denote a compact, connected Lie group
with maximal torus $T\subset G$. Assume that the spaces $Hom(\Z^n,
G)$ are all path connected. There is an isomorphism
$$H^*(B(2,G);\Q)\cong H^*(G/T\times BT;\Q)^{W(G)}$$
which is compatible with the well-known isomorphism $H^*(BG;\Q)\cong
H^*(BT;\Q)^{W(G)}$, where $W(G)$ denotes the Weyl group of $T$ in
$G$.
\end{thm}
\begin{proof}
Identify the space $G/T\times T^n$ with the product $G/T\times
B_n(\infty,T)$, where $G/T$ is seen as a constant simplicial space.
Thus the maps $\phi_{n}$ define a simplicial map $\phi$, the Weyl
group $W(G)$ acts simplicially on $G/T\times B_*(\infty, T)$ and
$\phi$ is invariant under this action. This yields the following
commutative diagram of simplicial spaces
$$\xymatrix{
G/T\times B_*(\infty, T)\ar[r]^\phi\ar[d] & B_*(2,G) \\
(G/T\times B_*(\infty, T))/W(G)\ar[ru]^{\overline{\phi}} & }.$$ The
simplicial spaces under consideration are all proper (note that this
was proved in \cite{AC} for $B_*(2,G)$) and so by a standard result
(see \cite{may2}) the rational cohomology isomorphisms induced by
the maps $\overline{\phi}_n$ induce an isomorphism between the
rational cohomology of the geometric realizations, yielding the
desired equivalence. The last assertion follows from the diagram
$$\xymatrix{
G/T\times BT\ar[r]^\phi\ar[d]^{proj} & B(2,G)\ar[d]^{inc} \\
BT\ar[r]^{inc} & BG }$$ which is commutative up to homotopy.
\end{proof}

\begin{rem}
If the spaces $Hom(\Z^n,G)$ are not path--connected, an analogous
result can be obtained using the generic component $Hom_0(\Z^n,G)$
i.e. the component of the trivial representation. This is indeed a
simplicial subspace, as the face and degeneracy maps are continuous
and preserve the trivial representation. Its realization gives rise
to a subspace $B(2,G)_0\subset B(2,G)$ for which the rational
cohomology can be computed as above.
\end{rem}

\begin{exm}\label{exm:general.properties}
In case $G=U(n)$, it was shown in \cite{AC} that the spaces of
ordered commuting $n$--tuples are all path--connected. Hence the
theorem applies here. Note that in this case the Weyl group is the
symmetric group $S_n$ which acts by permuting the entries in the
diagonal maximal torus. Recall that
$H^*(BT)\cong \Q [t_1,\dots , t_n]$,
a polynomial ring on $n$ two dimensional generators. Thus the
rational cohomology of $B(2,U(n))$ can be described as the ring of
invariants
$H^*(B(2,U(n));\Q)\cong (H^*(U(n)/T)\otimes
             \Q [t_1,\dots , t_n])^{S_n}$.
Note that the natural inclusion
$B(2,U(n))\to BU(n)$
has a stable section, as this is true for the map $BT\to BU(n)$, where $T\subset U(n)$ is a maximal torus
\cite{Snaith}. Thus $BU(n)$ is a stable retract of $B(q,U(n))$ for
all $q \geq 2$.
\end{exm}

\begin{rem}
The spaces $Hom(\mathbb Z^n,G)$ admit an action of $G$ by
conjugation, and the orbit space $Rep (\mathbb Z^n, G)$ can be
identified with the moduli space of flat $G$--bundles over the torus
$T^n$, an object of considerable geometric interest. The
structure of these spaces in the case of $G = U(n), SU(n), Sp(n)$ is
given in \cite{adem-cohen-gomez}.
\end{rem}

The next result in this section is a product decomposition arising
from looping the fibration $E(q,G) \to B(q,G) \to BG$. First
recall that there is a map $\iota_G:\Sigma(G) \to BG$ which is
natural for morphisms of topological groups obtained by identifying
the first filtration $F_1BG$ as $\Sigma(G)$. In addition, the
composite
\[
\begin{CD}
G  @>{E}>> \Omega \Sigma(G)@>{\Omega(\iota_G)}>>   \Omega BG
\end{CD}
\] is a homotopy equivalence where $E$ is the standard Freudenthal
suspension map and $G$ is a Lie group.

\begin{thm}\label{thm:product.decomposition}
If $G$ is a connected Lie group, then the associated looped
fibration
$$\Omega(E(q,G)) \to \Omega(B(q,G)) \to \Omega(BG)$$ has a
cross-section (up to homotopy). The cross-section is induced by a
map $\sigma(q,G): G \to \Omega(B(q,G))$ given by the composite

\[
\begin{CD}
G @>{E}>>  \Omega(\Sigma(G))  @>{\Omega(\iota(q,G))}>>
\Omega(B(q,G))
\end{CD}
\] where $\iota(q,G):\Sigma(G) \to B(q,G)$ is given by
the canonical identification $\Sigma(G) = F_1B(q,G)$ for $q \geq 2$,
and which is natural for morphisms in $G$. Thus there is a homotopy
equivalence $\theta(q,G):G \times \Omega(E(q,G)) \to \Omega
B(q,G)$.

Furthermore the cross-section $\sigma(q,G)$ and homotopy equivalence
$\theta(q,G)$ are natural in the sense that if $f:G \to H$ is a
morphism of topological groups, then there are strictly commutative
diagrams

\[
\begin{CD}
G @>{\sigma(q,G)}>> \Omega B(q,G)  @>{}>> \Omega BG \\
@V{f}VV            @VV{\Omega B(q,f)}V    @VV{\Omega Bf}V  \\
H @>{\sigma(n,H)}>> \Omega B(q,H)  @>{}>> \Omega BH.
\end{CD}
\] 

\[
\begin{CD}
G \times \Omega(E(q,G)) @>{\theta(q,G)}>> \Omega B(q,G)  \\
@V{ f \times \Omega E(q,f)}VV            @VV{\Omega B(q,f)}V     \\
H \times \Omega(E(q,H)) @>{\theta(q,H)}>> \Omega B(q,H).
\end{CD}
\]

\end{thm}

\begin{proof}

Recall that there is a canonical identification $\Sigma(G) =
F_1B(q,G)$ for $q \geq 2$ natural for morphisms in $G$. Thus there
is a natural factorization
\[
\begin{CD}
\Sigma(G) @>{}>> F_1B(2,G)  @>{}>> \cdots @>{}>> F_1(BG).
\end{CD}
\] together with the associated composite arising by taking adjoints

\[
\begin{CD}
G @>{E}>> \Omega\Sigma(G) @>{}>> \Omega F_1B(2,G) @>{}>> \cdots @>{}>>
\Omega(F_1(BG))@>{}>> \Omega(BG).
\end{CD}
\] So the following diagram strictly commutes:

\[
\begin{CD}
G @>{E}>> \Omega\Sigma(G) @>{}>> \Omega F_1B(2,G)@>{}>>
\cdots @>{}>> \Omega(F_1(BG))@>{}>> \Omega(BG)\\
@V{1}VV          @V{1}VV @VVV   @VVV      @VVV       @VV{}V \\
G @>{E}>> \Omega\Sigma(G) @>{}>> \Omega B(2,G)@>{}>> \cdots @>{}>>
\Omega(B(q,G))@>{}>> \Omega(BG)
\end{CD}
\]

\

The composite

\[
\begin{CD}
G @>{E}>> \Omega\Sigma(G) @>{}>> \Omega B(2,G) @>{}>> \cdots @>{}>>
\Omega(B(q,G))
\end{CD}
\] is the adjoint $\iota(q,G):\Sigma(G) \to B(q,G)$ and thus the induced map

\[
\begin{CD}
G @>{E}>> \Omega\Sigma(G) @>{}>> \Omega(B(q,G)) @>{}>> \Omega(BG)
\end{CD}
\] is a homotopy equivalence.

\

The looped fibration $\Omega(E(q,G)) \to \Omega(B(q,G)) \to
\Omega(BG)$ has a cross-section. The choice of homotopy equivalence
$\theta(q,G): G \times \Omega(E(q,G))\to \Omega(B(q,G))$ is the
composite
\[
\begin{CD}
G \times \Omega(E(q,G)) @>{\sigma(q,G) \times
\Omega(p(q,G))}>>\Omega(B(q,G)) \times \Omega(B(q,G))
@>{\hbox{multiply}}>> \Omega(B(q,G)).
\end{CD}
\] Thus the total space $\Omega B(q,G)$ is homotopy equivalent to a
product $G \times \Omega(E(q,G))$.

\

To finish, it suffices to check naturality. Notice that
  the maps $\sigma(q,G): G \to \Omega B(q,G)$ are natural for
morphisms in $G$;
 the maps $\Omega(p(q,G)): \Omega(E(q,G)) \to \Omega(B(q,G))$ are
natural for morphisms in $G$; and
the composite
\[
\begin{CD}
G \times \Omega(E(q,G)) @>{\sigma(q,G) \times
\Omega(p(q,G))}>>\Omega(B(q,G)) \times \Omega(B(q,G))
@>{\hbox{multiply}}>> \Omega(B(q,G))
\end{CD}
\] is natural for morphisms in $G$.
The theorem follows.
\end{proof}

\begin{rem}
The proof of Theorem \ref{thm:product.decomposition} does not 
require the hypothesis
that $G$ be connected. Namely, the fibration $B(q,G) \to BG$ is, 
\textbf{after looping}, 
always split for any $G$
(as seen in the proof above). Since a multiplicative fibration which is 
split is a trivial fibration, there is a homotopy
equivalence $G \times \Omega(E(q,G)) \simeq \Omega B(q,G)$.
This decomposition does not usually preserve the loop structure as the
following example shows. Consider
the group $G = A_5$: in this case, the space 
$\pi_0\Omega B(q,A_5) = \pi_1(B(2,A_5))$. On the other hand, there 
is an isomorphism of groups
$$\pi_0(A_5 \times \Omega(E(2,A_5)) = A_5 \times  F$$ 
where $F$ is a finitely generated free group of rank $854$.
Thus
$\pi_0(A_5 \times \Omega(E(2,A_5))$ and  $\pi_0\Omega B(q,A_5)$ 
fail to be isomorphic as groups, although they are isomorphic as sets.
\end{rem}

\section{The Homology of $B(q,G)$ when $G$ is Finite, and the 
Feit-Thompson Theorem}\label{sec:cohomology.Feit-Thompson}

The purpose of this section is to consider homological properties
of the spaces $E(q,G)$ and
$B(q,G)$ when $G$ is finite. In particular
the homology groups $H_*(E(q,G);\mathbb Z)$ are natural
$\mathbb Z G$--modules with potentially interesting properties.

If $K_*$ is a simplicial set, then define $\Z K_*$ as
the free abelian group on the simplices of $K_*$. The face maps of
$K_*$ define boundary maps by the equation
\[ \partial_n=\sum_{i=0}^n (-1)^i d_i :\Z K_n\to \Z K_{n-1}\]
and so $\Z K_*$ becomes a chain complex. The crux of this
construction are the isomorphisms
$H_*(\Z K_*)\cong H_*(Sing(|K_*|))\cong H_*(|K_*|)$
where $Sing(|K_*|)$ is the singular complex of $|K_*|$. Consider $\Z
B_*(q,G)$, which can be thought of as a subcomplex of $\Z
B_*(\infty,G)$; in fact
$\Z B_*(q,G)=\bigcup_{H\in \mathcal{N}_q(G)} \Z B_*(\infty,H)$.

Note that $\Z B_0(q,G)=\Z$ and $\Z B_1(q,G)=\Z[G]$ (the free abelian
group on $G$), so the chain complex looks like
\[\cdots\to \Z B_2(q,G)\stackrel{\partial_2}{\to} \Z[G]\stackrel{\partial_1}{\to} \Z\]
where $\partial_1=0$ and $\partial_2(x,y)=y-xy+x$. Define for each
$q\geq 2$ a subgroup of $\Z[G]$ by
\[I_q(G)=\langle y-xy+x| \Gamma^q(\langle x,y\rangle) =1, \mbox{ with } x,y\in G\rangle\]
Thus $I_2(G)\subseteq I_3(G)\subseteq\cdots\subseteq I_\infty(G)$,
where $\Gamma^q=\{ 1\}$ when $q=\infty$. Then
$H_1(B(q,G))=\Z[G]/I_q(G)$
and so there is a sequence of surjective maps
\[ H_1(B(2,G))\to H_1(B(3,G))\to\cdots\to H_1(BG)\]

\begin{cor}
If $G$ is a finite group and $q\geq 2$, then the $H_i(B(q,G);\Z)$
are finite abelian groups for all $i>0$, and their torsion only
occurs at primes dividing the order of $G$.
\end{cor}
\begin{proof}
As observed previously, $\Z B_*(q,G)$ can be thought of as the union
of the chain complexes generated by the maximal subgroups of class
$<q$. Therefore a Mayer-Vietoris spectral sequence can be used to
compute $H_*(B(q,G);\mathbb Z)$, involving the homology of finite
subgroups of $G$; hence their reduced homology is annihilated by
$|G|$. Thus the only torsion involved is at the primes dividing
$|G|$.
\end{proof}

Similarly one can consider $\Z E_*(q,G)$; in this case $\Z
E_0(q,G)=\Z$, and the complex looks like
\[ \cdots\to \Z E_2(q,G)\stackrel{\partial_2}{\to} \Z E_1(q,G)\stackrel{\partial_1}{\to} \Z[G]\]
where $\partial_1(a,x)=ax-a$ and
$\partial_2(a,x,y)=(ax,y)-(a,xy)+(a,x)$. As the natural projection
$E_*(q,G)\to B_*(q,G)$ is a simplicial map this yields a map of
chain complexes and hence one of homology groups. Note that this map
on $H_1$ takes the form
\[ H_1(E(q,G))\to H_1(BG),\,\,\,\,
(z,x)\mapsto x\]
at the chain level. As before, there is a sequence of surjective
maps \[ H_1(E(2,G))\to H_1(E(3,G))\to\cdots\to H_1(EG).\] The
following result shows that the first homology group contains
interesting information.

As has been discussed previously, there is a principal $G$-bundle
$E(q,G) \to B(q,G)$ which in the case of a discrete group gives rise
to a $G$--covering.
One question which arises is to analyze the structure of the
exact sequence of fundamental groups $1 \to \pi_1(E(q,G)) \to
\pi_1(B(q,G)) \to G \to 1$. In particular, this sequence gives a
natural representation $\rho: G \to Out(\pi_1(E(q,G))$; a pertinent
question here would be to ask for natural properties of this
representation. Passing to homology, the Serre spectral sequence
associated to the fibration $E(q,G)\to B(q,G)\to BG$ gives
rise to an exact sequence connecting the low dimensional
homology of these groups. Interestingly, this can be used to reformulate
the Feit--Thompson Theorem, which states that
every finite group of odd order is solvable.

\begin{prop}\label{prop:on.feit.thompson}
The Feit--Thompson Theorem is equivalent to the following result:
for $G$ a finite group of odd order, the homomorphism
$H_1(E(2,G);\mathbb Z)\to H_1(B(2,G);\mathbb Z)$
is not surjective.
\end{prop}

\begin{proof}
Consider the fibration
$E(2,G)\to B(2,G)\to BG$
and its associated 5--term exact sequence in homology; this yields
an exact sequence of the form
$$H_2(B(2,G); \mathbb Z)\to H_2(BG;\mathbb Z) \to H_1(E(2,G);\mathbb Z)_G\to H_1(B(2,G);\mathbb Z)\to
G/[G,G]\to 0.$$ The Feit--Thompson Theorem says that every odd order
group is solvable, which is equivalent to the condition that
$G/[G,G]\ne 1$ for all $G$ of odd order. This is precisely
equivalent to the failure of surjectivity for the map
$H_1(E(2,G);\Z)\to H_1(B(2,G);\Z)$.
\end{proof}

As noted previously $T(q)$ is the fundamental group of
$E(q,G)$, so $H_1(E(q,G);\mathbb Z)\cong T(q)/[T(q),T(q)]$. In the
sequel examples will be given where the geometry of the spaces
$E(q,G)$ will be explicitly determined (especially in the case
$q=2$) and the map on $H_1$ analyzed in some detail.

\section{$B(2,G)$ for Transitively Commutative Groups}\label{sec:TC.groups}

In this section the space $B(2,G)$ will be described for a very
particular class of finite groups, the \textsl{transitively
commutative groups}. 

\begin{defin} A transitively 
commutative\footnote{This definition appears in \cite{PY}, page 415. 
A related notion, that of 
a centralizer abelian or CA group,
refers to the situation when commutativity is transitive on all non--trivial
elements in the group.
Note that TC with trivial center is equivalent to non-abelian CA.}
or TC group $G$ is one satisfying
the following condition: 
given elements $g,h,k\in G-Z(G)$,
if $[g,h]=1=[h,k]$, then $[g,k]=1$.
\end{defin}

Note that the TC-groups are classified (see~\cite{Sc}, page 519 and Theorem 9.3.12); 
examples include abelian groups,
groups with an abelian normal subgroup of prime index, e.g. dihedral and generalized quaternion groups,
$SL(2,\F_{2^n})$, with $n\geq 2$, and 
all nonabelian groups of order $< 24$. The following elementary 
lemmas describe their structure (proofs are left to the reader).

\begin{lem}\label{TC}  Let $G$ be a nonabelian group; the following are all
equivalent
to $G$ being a TC group:
\begin{itemize}
\item[a)] If $g\notin Z(G)$, then $C(g)$ is abelian.
\item[b)] If $[g,h]=1$, then $C(g)=C(h)$ whenever $g,h\notin Z(G)$.
\item[c)] If $A,B\leq G$ and $Z(G)<C_G(A)\leq C_G(B)<G$, then $C_G(A)=C_G(B)$.
\end{itemize}
\end{lem}

\begin{lem}\label{TC2} Let $G$ be a TC group with trivial center.
\begin{itemize}
\item[a)] The Sylow subgroups of $G$ are abelian and intersect trivially.
\item[b)] $\{ C_G(x)|\ x\not\in Z(G)\}$ is the family of maximal abelian subgroups of $G$.
\item[c)] The maximal abelian subgroups of $G$ intersect trivially.
\item[d)] If $H$ is a maximal abelian subgroup and $P\in Syl_p(H)$, then $P\in Syl_p(G)$.
\end{itemize}
\end{lem}

Now let $a_1,\ldots,a_k\in G-Z(G)$ be a set of representatives of their
centralizers so that
$G=\bigcup_{1\leq i\leq k} C_G(a_i)$ and no smaller number of centralizers 
covers $G$. Note that
Lemma~\ref{TC} shows that the groups defining this union do not
depend on the choice of representatives. This number $k$ is called
the {\it number of centralizers that cover $G$}. Moreover, note that
each $C(a_i)$ is a maximal abelian subgroup of $G$ and that distinct
centralizers intersect along the center $Z(G)$.
Applying Theorem~\ref{thm:aspherical B(q,G)}
in this situation yields the following

\begin{prop}\label{main} If $G$ is a TC group, then $B(2,G)\simeq BG(2)$, where
$G(2)$ is defined in \ref{colimit} and as a consequence
of \ref{TC} is the amalgamated product of the maximal abelian subgroups
of $G$ along the center of $G$. In particular, $E(2,G)$ is a
$K(\pi,1)$ as well.
\end{prop}

\begin{cor}\label{cor:stable.splitting.for.TC} If $G$ is a $TC$ group with trivial center, then
$$B(2,G)\simeq \bigvee_{1\leq i\leq k} \left( \prod_{p\mid |C_G(a_i)|}  BP \right)$$
where $P\in Syl_p(G)$.
\end{cor}
\begin{proof}
$B_*(2,G)$ is the one-point union of the simplicial spaces
$B_*(2,C_G(a_i))$, and it's easy to see that
$C_G(a_i)\cong\prod_{p\mid |C_G(a_i)|}  P$
with $P\in Syl_p(G)$. The result follows.
\end{proof}

\begin{prop}\label{prop:stable.splitting.sylows.for.TC}
If $G$ is a finite $TC$ group with trivial center, then there is a
stable homotopy equivalence
$B(2,G)\simeq \bigvee_{p\mid |G|} \bigvee_{P\in Syl_p(G)} BP$.
\end{prop}

\begin{exm}
Let $G=SL_2(\mathbb F_8)$, this is a TC group of order $504=2^3\cdot
3^2\cdot 7$, with trivial center and with $p$--Sylow subgroups
$(\mathbb Z/2)^3$, $\mathbb Z/9$ and $\mathbb Z/7$. There is a stable
homotopy equivalence
$$B(2, SL_2(\mathbb F_8))\simeq
\bigvee^9 B(\mathbb Z/2)^3\bigvee[\bigvee^{28} B\mathbb Z/9]\bigvee
[\bigvee^{36} B\mathbb Z/7]$$
\end{exm}

Consider the structure of $E(2,G)$ for $G$ a finite $TC$ group. The
main ingredient which can be applied is the theory of trees, as
described in \cite{Serre}. In this situation, the colimit group
$G(2)=*_{Z(G)}C_G(a_i)$ acts on a graph $X$ by setting
$$\rm{Edges}(X)=\coprod^k G(2)/Z(G),\,\,\,\,
\rm{Vertices}(X)=G(2)/Z(G)\sqcup~ [\coprod_{1\leq i\leq k} G(2)/C(a_i)]$$
and where the $i^{th}$ copy of $G(2)/Z(G)$ is identified with the
vertices $G(2)/Z(G)\sqcup G(2)/C(a_i)$ by using the identity and the
natural maps $G(2)/Z(G)\to G(2)/C(a_i)$. It turns out that $X$ is a
tree (see~\cite{Serre} p.~38), with a simplicial action of $G(2)$
having as fundamental domain the \textsl{tree of groups} given by
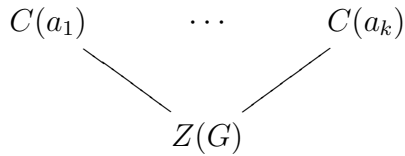
\[\xymatrix{
C(a_1)\ar@{^{}}[rd] & \cdots & C(a_k)\ar@{^{}}[ld] \\
 & Z(G) & }\]

\begin{prop}\label{prop:TC.tree} If $G$ is a TC group then the fundamental group of $E(2,G)$ is a free group of rank
$$N_G=1-|G:Z(G)|+\left( \sum_{1\leq i\leq k} |G:Z(G)|-|G:C_G(a_i)|\right)$$
and so $E(2,G)\simeq \bigvee_{N_G} S^1.$
\end{prop}
\begin{proof}
Notice that the fundamental group of $E(2,G)$ is the kernel of the
homomorphism between $G(2)$ to $G$ induced by the natural
inclusions, and no element in $T(2)=\pi_1(E(2,G))$ is conjugate to
an element in a centralizer or else the image of this element is
nontrivial in $G$. Hence by a theorem due to Kurosh (see for instance Corollary A.2
from~\cite{Br} or~\cite{Serre} p.~56) it must be a free group. 
To find the rank it suffices to calculate the Euler
characteristic of the groups in the extension
$1\to T(2)\to G(2)\to G\to 1$
On the one hand $\chi (G(2)) = \frac{\chi (T(2))}{|G|}$ and so
$\rm{rank}~(T(2))=1-\chi(T(2))=1-\chi(G(2))|G|$. On the other hand,
since $G(2)$ is the amalgamated product of the centralizers covering
$G$ along $Z(G)$, it follows by an inductive argument that
\[ \chi(G(2))=\frac{1}{|Z(G)|}+\sum_{i=1}^{k} [\frac{1}{|C_G(a_i)|}-\frac{1}{|Z(G)|}]\]
\[rank~T(2)=1-\chi(G(2))|G|=
 1 - |G|\left( \frac{1}{|Z(G)|}+\sum_{1\leq i\leq k} [\frac{1}{|C_G(a_i)|}-\frac{1}{|Z(G)|}]\right) \]
and the result follows.
\end{proof}
Note that $Y= X/T(2)$ is a finite graph with a $G$--action. Looking
at the cellular chain complex of $Y$ yields the following exact
sequence of $\mathbb Z G$-modules
\[0\to H_1(E(2,G))\to \bigoplus^k\Z[G/Z(G)]\xrightarrow{\phi}\Z[G/Z(G)]\oplus
\bigoplus_{1\leq i\leq k}\Z[G/C(a_i)]\to \Z\to 0 \eqno (I)\] where
$\Z[G/H]$ denotes the usual permutation module with isotropy $H$,
and the structure of $H_1(E(2,G))$ as a $G$-module is the one
determined by the extension $$1\to \pi_1(E(2,G))\to G(2)\to G\to 1.$$
The map $\phi$ is described as follows; let $j_i:\mathbb Z[G/Z(G)]
\to \mathbb Z[G/C(a_i)]$ be the natural projection induced by the
inclusion of $Z(G)$ in $C(a_i)$ for $i=1,\dots , k$. Then
$$\phi (v_1, v_2,\dots , v_k) = (\sum_{i=1}^k v_i, -j_1(v_1),
\dots , -j_k(v_k))$$

There are two long exact sequences associated to this exact
sequence, obtained by applying $G$--hypercohomology to it (see
\cite{AM}, Chapter V). The first one is:

$$\dots \to H_i(G;\mathbb Z)\to H_{i-2}(G,H_1(E(2,G)))\to
H_{i-1}(B(2,G);\mathbb Z)\to H_{i-1}(G;\mathbb Z)\to\dots \eqno
(II)$$ 
The other one is of the form
$$\dots \to H_{q+1}(B(2,G))\to \bigoplus^k H_q(Z(G))
\xrightarrow{\phi_*} H_q(Z(G))\oplus [\bigoplus_{i=1}^k H_q(C(a_i))]
\to H_q(B(2,G)) \to\dots \eqno (III)$$ 
If $A$ and $B$ are finite abelian groups and $A\subset B$, then
$H_*(A)\subset H_*(B)$, whence it follows that $\phi_*$ is
injective, and so for each $q\ge 0$, $H_q(B(2,G))$ can be described
via the short exact sequence
$$0\to \bigoplus^k H_q(Z(G)) \xrightarrow{\phi_*} H_q(Z(G))
\oplus \bigoplus_{i=1}^k H_q(C(a_i)) \to H_q(B(2,G))\to 0.$$ 

Now consider the special case when $G$ is a $TC$ group with a
trivial center. In this situation the chain group $C_1(Y)$ is a free
$\mathbb ZG$--module, and the $p$--Sylow subgroups of the
centralizers $C_G(a_i)$ are $p$--Sylow subgroups for $G$.
\begin{lem}
Let $M$ denote a permutation $\mathbb Z G$--module of finite rank
and $\epsilon : M\to\mathbb Z$ the usual augmentation map onto the
trivial module. Then the map $\epsilon$ splits over $\mathbb ZG$ if
and only if $|G|$ divides the least common multiple of the orders of
the isotropy subgroups.
\end{lem}
\begin{proof}
We can write $M = \bigoplus_{i=1}^k \mathbb Z[G/H_i]$ where the
$H_i\subset G$ are subgroups. If the map $\epsilon$ splits, then
$\mathbb Z/|G|$ is a direct summand in $\widehat{H}^0(G,M)\cong
\bigoplus_{i=1}^k \mathbb Z/|H_i|$ hence $|G|$ divides the l.c.m. of
the orders of the subgroups $H_1,\dots, H_k$. Conversely, suppose
that $|G|$ divides this l.c.m. This means that it is possible to
find a summand $\mathbb Z/|G|\subset \bigoplus_{i=1}^k \mathbb
Z/|H_i|$ which maps bijectively onto $\widehat{H}^0(G;\mathbb Z)$
under the map induced in cohomology by $\epsilon$. Using the main
result in \cite{adem}, there exists a trivial submodule of rank one
$T\subset M$ which represents this class of highest exponent and
splits off as a direct summand of $M$; this defines the desired
splitting.
\end{proof}

Consider the case when $G$ is a TC group with trivial center. The
kernel of the augmentation map $\epsilon: C_0(Y)\to \mathbb Z$ is
isomorphic to the quotient module $C_1(Y)/H_1(E(2,G))$, where $C_1(Y)$ is
a free $\mathbb ZG$ module. Applying the previous result yields a
$G$--splitting:
$C_0(Y)\cong \mathbb Z \oplus [C_1(Y)/H_1(E(2,G))]$.
For the statement of the following corollary recall that given $M$ a
$\mathbb ZG$--module, a finitely generated projective module $F$,
and a surjection $f:F\to M$, then the kernel of $f$ is uniquely
defined up to projective summands, and is denoted $\Omega^1(M)$ (see
\cite{AM}, Chapter II).
Similarly if $N\to F'$ is a monomorphism and $F'$ is
projective, then the cokernel is uniquely defined up to projective
summands and denoted $\Omega^{-1}(N)$.

\begin{cor}
If $G$ is a $TC$ group with trivial center, then $H_1(E(2,G))$ is a
$\mathbb Z$--torsion free module which up to projective factors is
isomorphic to $\Omega^1(I)$, where $I$ denotes the kernel of the
augmentation map $C_0(Y)\to \mathbb Z$. Furthermore up to
projective summands there is a
splitting of $\mathbb ZG$--modules
$$\Z[G/Z(G)]\oplus
\bigoplus_{1\leq i\leq k}\Z[G/C(a_i)]
\cong \mathbb Z\oplus\Omega^{-1}(H_1(E(2,G))).$$ 
\end{cor}

\begin{exm}
$G=A_5$, the alternating group on five letters is a $TC$
group with trivial center. 
In this case $H_1(E(2,A_5))$ is a module of rank equal to $854$
and
$$H_i(B(2,A_5))\cong [H_i(\mathbb Z/2\times\mathbb Z/2)]^5
\oplus [H_i(\mathbb Z/3)]^{10} \oplus [H_i(\mathbb Z/5)]^6$$
which in turn is isomorphic to
$H_{i-1}(A_5,H_1(E(2,A_5)))\oplus H_i(A_5)$.
\end{exm}

\begin{exm}
Let $G=Q_8$, the quaternion group of order eight. In this case the
center is $\mathbb Z/2$, and there are three maximal abelian
subgroups (each of order four) intersecting along this central
subgroup. In this case $B(2,Q_8)\cong B(*^3_{\mathbb Z/2}\mathbb
Z/4)$, the amalgamation of the three groups along the common
$\mathbb Z/2$. The module $H_1(E(2,Q_8))$ is of rank equal to three. It
fits into an exact sequence
$$0\to H_1(E(2,Q_8))\to \mathbb Z[Q_8/\mathbb Z/2]^3
\to \mathbb Z [Q_8/\mathbb Z/2] \oplus \bigoplus_{i=1}^3
Z[Q_8/C(a_i)] \to\mathbb Z\to 0$$ and $H_i(B(2,Q_8))$ can be
computed as
\[ H_i(B(2,Q_8); \Z)
\cong\left\{\begin{array}{r@{\quad\hbox{if}\quad}l}
\Z & i=0\\
\mathbb Z/4\oplus\mathbb Z/2\oplus\mathbb Z/2  & i>0~\rm{odd}\\
0 & i>0~\rm{even}
\end{array}\right.\]
\end{exm}

Note that the sequence (I) also allows us to compute the character
of the representation
$G\to Aut(H_1(E(2,G)\otimes\mathbb C)$,
which will be denoted by $\mathcal{X}_{E(2,G)}$. 
If $G$ is a TC group the kernel of the
character $\mathcal{X}_{E(2,G)}$ is precisely the center of $G$
and 
thus
$G/Z(G)\to Aut(H_1(E(2,G)))$
is a faithful representation.
This result shows that the representation $G\to
Out(\pi_1(E(2,G)))$ is faithful when $G$ is a TC group with trivial
center; moreover, the representation $\pi_1(B(2,G))\to
Aut(\pi_1(E(2,G)))$ is faithful as well.

\end{document}